\documentclass[reqno]{amsart}%amsart

\usepackage{color}
\usepackage[foot]{amsaddr}
\usepackage{booktabs}
\usepackage{bigstrut}
\usepackage{amsmath,amsthm,bm}
\usepackage{amsfonts}
\usepackage{algorithmic}
\usepackage{algorithm}
\usepackage{multirow}
\usepackage{overpic}
  \usepackage{paralist}
  \usepackage{graphics} %% add this and next lines if pictures should be in esp format
  \usepackage{epsfig} %For pictures: screened artwork should be set up with an 85 or 100 line screen
\usepackage{graphicx}  \usepackage{epstopdf} %This is to transfer .eps figure to .pdf figure; please compile your article using PDFLaTex or PDFTeXify.

\newtheorem{theorem}{Theorem}[section]
\newtheorem{corollary}[theorem]{Corollary}

\newtheorem{lemma}[theorem]{Lemma}

\theoremstyle{definition}
\newtheorem{definition}[theorem]{Definition}
\newtheorem{remark}[theorem]{Remark}

\DeclareMathOperator*{\esssup}{ess\,sup}

\graphicspath{{pics/}}

\begin{document}

\title[]{Continuum limit of $p$-biharmonic equations on graphs}

\author{Kehan Shi $^{\ast}$, Martin Burger $^{\dagger,\ddagger}$}
\address{$^{\ast}$Department of Mathematics, China Jiliang University,
Hangzhou 310018, China}
\address{$^{\dagger}$Computational Imaging Group and Helmholtz Imaging, Deutsches Elektronen-Synchrotron DESY, 22607 Hamburg, Germany}
% \email{kshi@cjlu.edu.cn}
\address{$^{\ddagger}$Fachbereich Mathematik, Universit\"{a}t Hamburg, 20146 Hamburg, Germany}

% 2020 MSC numbers are required.
% \subjclass{35G20, 35R02, 05C65}
% Please provide a minimum of 5 keywords or phrases.

\begin{abstract}
This paper studies the $p$-biharmonic equation on graphs,
which arises in point cloud processing and can be interpreted as a natural extension of the graph $p$-Laplacian from the perspective of hypergraphs.
The asymptotic behavior of the solution is investigated for the random geometric graph when the number of data points goes to infinity.
We show that the continuum limit is an appropriately weighted  $p$-biharmonic equation with homogeneous Neumann boundary conditions.
The result relies on the uniform $L^p$ estimates for solutions and gradients of nonlocal and graph Poisson equations. The $L^\infty$ estimates of solutions are also obtained as a byproduct.
\end{abstract}

\maketitle

%%%%%%%%%%%%%%%%%%%%%%%%%%%%%%%%%%%%%%%%%%%%%%%%%%%%%%

\section{Introduction}
In many problems in machine learning and data processing, we are given a discrete data set $\Omega_n=\{x_i\}_{i=1}^n$ with $n$ data points drawn from a bounded set $\Omega\subset\mathbb{R}^d$ and required to achieve tasks such as  clustering, label spreading, or denoising of the data set.
It is generally assumed that the data set and its geometric structure are  represented by
a graph $G_n=(V_n, E_n, W_n)$ with vertices $V_n=\Omega_n$, edges $E_n$, and weights $W_n$.
An edge $e_{i,j}\in E_n$ connecting two vertices $x_i$ and $x_j$ exists if their distance is sufficiently small. The weight $W_{i,j}\in W_n$ for edge $e_{i,j}$ is a nonincreasing function of the distance.

Notions and algorithms in the continuum setting have been generalized to the graph setting for such problems.
A typical one is the graph $p$-Laplacian regularization \cite{zhou2005regularization,elmoataz2015p}
\begin{equation}\label{eq:1.1}
  \frac{1}{n^2\varepsilon_n^p}\sum_{i,j=1}^nW_{i,j}|u(x_i)-u(x_j)|^p,\quad 1\leq p<\infty,
\end{equation}
for a function $u:\Omega_n\rightarrow\mathbb{R}$, where $\frac{1}{n^2\varepsilon_n^p}$ is the rescaling parameter and $\varepsilon_n$ denotes the connection radius of graph $G_n$.
When $p=2$, it is the well-known Laplacian regularization and has attracted a lot of attention \cite{zhou2005learning,von2007tutorial,gilboa2007nonlocal}.
Different $p$ are preferred for different applications. For example, if $p=1$, it becomes the total variation on graphs, which is the generalization of the graph cut and is particularly suitable for clustering \cite{garcia2016continuum}.
While in semi-supervised learning, $p>d$ is suggested to prevent the occurrence of spikes \cite{el2016asymptotic,slepcev2019analysis}.
The asymptotic behavior of functional \eqref{eq:1.1} as the number of data points goes to infinity has been studied in \cite{garcia2016continuum,slepcev2019analysis} to reveal its connection with the classical $p$-Laplacian regularization.
Since the property of the latter is clear, it  gives us an insight into functional \eqref{eq:1.1} and is crucial in applications.

% It $\Gamma$-converges to the continuum $p$-Laplacian regularization (with weight $\rho$ being the density of the distribution related to the data set $\Omega_n$) in an appropriate topology as $n\rightarrow \infty$.

Recently, higher-order regularization on graphs has been proposed.
In \cite{dong2020cure}, the authors suggested to utilize the graph Laplacian
\begin{equation}\label{eq:1.11}
  \Delta_{G_n} u(x_i)=\frac{1}{n\varepsilon_n^2}\sum_{j=1}^n W_{i,j}(u(x_j)-u(x_i)),\quad x_i\in\Omega_n,
\end{equation}
in the regularizer and interpreted it as an approximation of the curvature regularization on the manifold.
It is a special case of minimizing
\begin{equation}\label{eq:1.2}
  \frac{1}{pn}\sum_{i=1}^n
  \left|\Delta_{G_n} u(x_i)\right|^p+\frac{\lambda}{2}\sum_{i=1}^n|u(x_i)-f(x_i)|^2,
\end{equation}
when $p=2$ and $\lambda=0$.
Here the graph Laplacian $\Delta_{G_n}$ is negative semi-definite, whereas in the machine learning community it is usually defined as positive semi-definite.
 Functional \eqref{eq:1.2} with general $p>1$ and $\lambda>0$ was considered in \cite{el2020discrete} for point cloud denoising, where $f$ denotes the observed point cloud data with noise.
One benefit of utilizing the higher-order regularization in the continuum setting is that it gains more smoothness than the first-order regularization \cite{lysaker2003noise,bredies2010total}. The same property is expected in the graph setting. To verify it we need to show rigorously that functional \eqref{eq:1.2} is an approximation of a continuum higher-order model.

% The existence and uniqueness of solution for the associated Euler--Lagrange equation on a fixed graph was discussed.

The goal of this paper is to study the asymptotic behavior for the solution of the $p$-biharmonic equation
\begin{align}\label{eq:1.4}
    -\Delta_{G_n}\left(|\Delta_{G_n}u|^{p-2}\Delta_{G_n}u\right)(x_i)
    +\lambda(f(x_i)-u(x_i))=0,&\quad x_i\in\Omega_n,
\end{align}
 on a random geometric graph $G_n$ as $n\rightarrow\infty$, where the data $\Omega_n$ are drawn from a given measure $\mu$ with density $\rho$ and $1<p<\infty$.
Although the continuum limit of PDEs on graphs is considered, the tools we use follow from \cite{garcia2016continuum,slepcev2019analysis}, in which the continuum limit of functional \eqref{eq:1.1} was studied in the $TL^p$ metric via optimal transportation.
The reason we study equation \eqref{eq:1.4} instead of the associated variational model \eqref{eq:1.2} is that
we can treat it as an elliptic system and utilize results for elliptic equations on graphs.
It also inspires us to study more general PDE systems on graphs for data analysis in the future.
The $L^p$ \textit{a priori} estimates for solutions and gradients of graph Poisson equations play a key role.
The uniform $L^\infty$ estimate for the solution is also established as a corollary.

PDEs on graphs can be rewritten as nonlocal equations via the transportation map. The latter have been widely used as the approximation of classical PDEs \cite{andreu2010nonlocal}.
The nonlocal version of equation \eqref{eq:1.4} reads
\begin{align}\label{eq:1.5}
  -\Delta^\eta_{\rho,\varepsilon}\left(|\Delta^\eta_{\rho,\varepsilon}u|^{p-2}\Delta^\eta_{\rho,\varepsilon}u\right)(x)
  +\lambda(f(x)-u(x))=0,&\quad x\in\Omega,
\end{align}
where
\begin{equation*}
  \Delta^\eta_{\rho,\varepsilon}u(x)=\frac{1}{\varepsilon^2}\int_\Omega{\eta}_{\varepsilon}
(|x-y|)(u(y)-u(x))\rho(y)dy, \quad \varepsilon>0,
\end{equation*}
denotes the weighted nonlocal Laplacian and $\eta_\varepsilon(s)=\frac{1}{\varepsilon^d}\eta(s/\varepsilon)$ is a rescaled nonlocal kernel.
Equation \eqref{eq:1.5} acts as a bridge between equation \eqref{eq:1.4} and its continuum analogue, i.e., the classical $p$-biharmonic equation
\begin{align}\label{eq:1.6}
  -\Delta_\rho\left(|\Delta_\rho u|^{p-2}\Delta_\rho u\right)(x)
  +\lambda(f(x)-u(x))=0,&\quad x\in\Omega,
\end{align}
where
\begin{equation*}
  \Delta_\rho u(x)=\frac{\sigma_\eta}{2\rho}\textrm{div}(\rho^2\nabla u),
\end{equation*}
is the weighted Laplacian and $\sigma_\eta$ is a constant.
In fact, we first establish \textit{a priori} estimates for solutions of nonlocal Poisson equations and equation \eqref{eq:1.5}, then generalize the results to graph Poisson equations and equation \eqref{eq:1.4}.
We mention that the nonlocal biharmonic equation (i.e., equation \eqref{eq:1.6} with $p=2$ and $\rho\equiv 1$) has been considered for image processing \cite{wen2023nonlocal}.
The asymptotic behavior for solutions of nonlocal Dirichlet and Navier boundary value problems have been studied in \cite{radu2017nonlocal,shi2022nonlocal}.
 No boundary condition is explicitly posed for equation \eqref{eq:1.5}, which means that equation \eqref{eq:1.6} is with homogeneous Neumann boundary conditions.

This paper is organized as follows. In the rest of this section, we
give a new model explanation  for equation \eqref{eq:1.4} from the perspective of hypergraph $p$-Laplacians.
Section 2 presents assumptions and mathematical tools needed for the study.
Section 3 proves the consistency of the nonlocal/graph Laplacian and the classical Laplacian for sufficiently smooth functions  with homogeneous Neumann boundary conditions.
In Section 4, we establish $L^p$ and $L^\infty$ \textit{a priori} estimates for  nonlocal and graph Poisson equations.
The main result is given in Section 5.
We show the continuum limit of equation \eqref{eq:1.4} when the number of data points goes to infinity.

\subsection*{The $p$-Laplacian equation on hypergraphs}
A hypergraph is a natural extension of a graph in which an edge (also known as hyperedge in the unoriented hypergraph or hyperarc in the oriented hypergraph) can connect to a subset of vertices. It allows us to represent higher-order relationships in data.
Generalizing graph-based methods to hypergraphs has led to improvements in applications \cite{zhou2006learning,hein2013total}.
Besides, it is possible to define differential operators on hypergraphs in a similar way to the graph case and study their properties \cite{jost2019hypergraph,fazeny2023hypergraph}.

The construction of a hypergraph from a given data set is nontrivial and affects the performance in applications. A brief review can be found in \cite{gao2020hypergraph}.
We shall utilize the distance-based $\varepsilon$-neighborhood method in which each hyperedge or hyperarc is a subset consisting of a vertex and its neighbors in the $\varepsilon$-ball.

Let $H=(V, A)$ be an oriented hypergraph with vertices $V$ and hyperarcs $A$.
Each hyperarc $a\in A$ is a pair of elements of the form
\begin{equation*}
  a=(a^{out}, a^{in}),
\end{equation*}
where $a^{out}$ and $a^{in}$ are nonempty and disjoint subsets of $V$ that represent the output and input vertices of the hyperarc.
Scalar products for functions defined on vertices and for functions defined on hyperarcs are given as follows.
\begin{definition}\label{de1.1}
  For $f, g: V\rightarrow\mathbb{R}$, let
  \begin{equation*}
    (f,g)_{V}=\sum_{v\in V}f(v)\cdot g(v).
  \end{equation*}
  For $F, G: A\rightarrow\mathbb{R}$, let
  \begin{equation*}
    (F,G)_{A}=\sum_{a\in A}F(a)\cdot G(a).
  \end{equation*}
\end{definition}
Similar to the oriented graph, we define the gradient operator and its adjoint operator on hypergraphs.
\begin{definition}\label{de1.2}
  For $f: V\rightarrow\mathbb{R}$ and $a\in A$, let
  \begin{equation*}
    \nabla_H f(a)=\sum_{v\in a^{in}}f(v)
    -\frac{|a^{in}|}{|a^{out}|}\sum_{v\in a^{out}}f(v).
  \end{equation*}
  For $G: A\rightarrow\mathbb{R}$ and $v\in V$, let
  \begin{equation*}
    \nabla^*_H G(v)=\sum_{\{a: v\in a^{in}\}}G(a)
    -\sum_{\{a: v\in a^{out}\}}\frac{|a^{in}|}{|a^{out}|}G(a).
  \end{equation*}
\end{definition}
Here we presented one possible definition for the hypergraph gradient for our purposes. Other definitions are also valid and may be of interest in specific problems \cite{jost2019hypergraph,fazeny2023hypergraph}.

It follows from Definition \ref{de1.1}--\ref{de1.2} that
\begin{equation*}
  (\nabla_H f, G)_A=(f, \nabla^*_H G)_V,
\end{equation*}
i.e., $\nabla^*_H$ is the adjoint operator of $\nabla_H$.
Then the divergence operator is given as $$\textrm{div}_H=-\nabla^*_H$$ and the hypergraph Laplacian  $$\Delta_H f=\textrm{div}_H (\nabla_H f),$$ for a function $f: V\rightarrow\mathbb{R}$.

Now let us construct a hypergraph $H_n=(V_n,A_n)$ by choosing $V_n=\Omega_n=\{x_i\}_{i=1}^n$ and $A_n=\{a_i\}_{i=1}^n$, where
\begin{equation*}
  a_i=(a_i^{out}, a_i^{in}),
\end{equation*}
and $a_i^{out}=\{x_i\}$ and $a_i^{in}=\{x_j\in\Omega_n\backslash \{x_i\},
\textrm{dist}(x_i,x_j)<\varepsilon_n \}$. We recall that $\varepsilon_n$ is the connection radius of graph $G_n$ that appears in functional \eqref{eq:1.1}.
Every vertex $x_i$ corresponds to a hyperarc $a_i$ through its output $a_i^{out}$ and vice versa.

The aforementioned differential operators on hypergraph $H_n$ are as follows.
For $f: \Omega_n\rightarrow\mathbb{R}$ and $a_i\in A_n$,
\begin{equation*}
  \nabla_{H_n} f(a_i)=\sum_{x_j\in a_i^{in}}f(x_j)-|a_i^{in}|f(x_i)=\sum_{j=1}^{n}W_{i,j}'(f(x_j)-f(x_i)),
\end{equation*}
where
\begin{align*}
  W'_{i,j}=
  \begin{cases}
    1, &\textrm{if } \textrm{dist}(x_i,x_j)<\varepsilon_n, \\
    0,  &\textrm{otherwise},
  \end{cases}
\end{align*}
is symmetric.
And
for $G: A_n\rightarrow\mathbb{R}$ and $x_i\in V_n$,
  \begin{equation*}
    \nabla^*_{H_n} G(x_i)=\sum_{\{a_j: x_i\in a_j^{in}\}}G(a_j)
    -|a_i^{in}|G(a_i)
    =\sum_{j=1}^{n}W_{i,j}'(G(a_j)-G(a_i)).
  \end{equation*}
By direct calculation, we have the $p$-Laplacian on $H_n$
\begin{align}\label{eq:hyper}
  \textrm{div}_{H_n} (|\nabla_{H_n} f|^{p-2} \nabla_{H_n} f)(x_i)
  =-\Delta_{G'_n}(|\Delta_{G'_n} f|^{p-2}\Delta_{G'_n} f)(x_i),\quad x_i\in\Omega_n,
\end{align}
where we ignore the rescaling parameter $\frac{1}{n\varepsilon_n^2}$ in $\Delta_{G'_n}$ (defined in \eqref{eq:1.11}) and graph $G'_n=(V_n,E_n,W_n':=\{W'_{i,j}\})$ is a simplified version of graph $G_n=(V_n,E_n,W_n)$ with weight one being given to every edge.

Equality \eqref{eq:hyper} establishes the connection between the $p$-biharmonic operator on graphs
and the $p$-Laplacian operator on hypergraphs. This provides a new model explanation for equation \eqref{eq:1.4}. Meanwhile, the asymptotic behavior we obtain for equation \eqref{eq:1.4} is also valid for the hypergraph $p$-Laplacian equation. It should be noticed
that the conclusion applies only to hypergraphs constructed by the $\varepsilon_n$-neighborhood method.
For a general hypergraph, hyperedges can be generated by different methods and differential operators on it may have different definitions. This brings difficulties for studying the asymptotic behavior of PDEs on general hypergraphs.

\section{Preliminaries}
\subsection{Settings}
Let
$
  \Omega_n=\{x_1, x_2,\cdots, x_n\}
$
be a point cloud in a bounded domain $\Omega\subset\mathbb{R}^d$.
The graph-based approach for data processing constructs a weighted graph
$G_n=(V_n, E_n, W_n)$ with vertices $V_n=\Omega_n$, edges $E_n$, and weights $W_n$.
We consider the $\varepsilon_n$-neighborhood graph, in which an edge $e_{i,j}\in E_n$ connecting $x_i$ and $x_j$ exists if $\textrm{dist}(x_i,x_j)<\varepsilon_n$.
Here $\varepsilon_n>0$ is the connection radius depending on $n$.
The weight for edge $e_{i,j}$ is determined by the distance between  $x_i$ and $x_j$ and is set as
\begin{equation*}
  W_{i,j}=\eta_{\varepsilon_n}(|x_i-x_j|),
\end{equation*}
where $\eta_{\varepsilon_n}(s)=\frac{1}{\varepsilon_n^d}\eta(s/\varepsilon_n)$ and
 $\eta$ is a {nonincreasing kernel with compact support.}

Let us consider the $p$-biharmonic equation on graph $G_n$
\begin{align}\label{P}
    -\Delta_{G_n}(|\Delta_{G_n}u|^{p-2}\Delta_{G_n}u)(x_i)+\lambda(f(x_i)-u(x_i))=0, \quad x_i\in \Omega_n,
\end{align}
where $1<p<\infty$ and
\begin{equation*}
  \Delta_{G_n}u(x_i)=\frac{1}{n\varepsilon_n^2}\sum_{j=1}^n\eta_{\varepsilon_n}
  (|x_i-x_j|)(u(x_j)-u(x_i))
\end{equation*}
denotes the graph Laplacian of $u: \Omega_n\rightarrow \mathbb{R}$ at $x_i\in\Omega_n$.

In this paper, we let the number of data points $n\rightarrow\infty$ and study the asymptotic behavior for the solution of equation \eqref{P}.
% More precisely, we shall prove that the equation is asymptotically consistent with the classical $p$-biharmonic equation with an appropriate weight.
The assumptions are listed as follows.
\begin{enumerate}
  \item [(A1)] $\Omega\subset\mathbb{R}^d$ $ (d\geq 2)$ is a connected and bounded open domain with sufficiently smooth boundary $\partial\Omega$.
  \item [(A2)]
  $f$ is a bounded and continuous function on $\Omega$.
  \item [(A3)]  $\eta: [0,\infty)\rightarrow [0,\infty)$ is a {nonincreasing piecewise Lipschitz function with compact support.} Without loss of generality, we assume that $\eta(1)\geq 1$ and $\eta(2)=0$.
  \item [(A4)]  The connection radius $\varepsilon_n\rightarrow 0$ as $n\rightarrow\infty$ and has the lower bound
  \begin{align*}
  \varepsilon_n\gg \sqrt{\delta_n},
  \end{align*}
  where
  \begin{align}\label{delta}
    \delta_n=\frac{(\ln n)^{1/d}}{n^{1/d}}.
    \end{align}
  % \begin{align}\label{delta}
  % \delta_n=
  %   \begin{cases}
  %     \sqrt{\frac{\ln\ln(n)}{n}}, & \mbox{if } d=1, \\
  %     \frac{(\ln n)^{3/4}}{\sqrt{n}}, & \mbox{if } d=2, \\
  %     \frac{(\ln n)^{1/d}}{n^{1/d}}, & \mbox{if  } d\geq 3.
  %   \end{cases}
  % \end{align}
  \item [(A5)] The data points $\{x_i\}_{i=1}^n$ are  independent random samples of a probability measure $\mu$ on $\Omega$.
  \item [(A6)] The density $\rho$ of $\mu$ (with respective to the Lebesgue measure) is $C^1(\overline{\Omega})$ with positive lower and upper bounds $0< \inf_{x\in\Omega} \rho(x)\leq \sup_{x\in\Omega}\rho(x)<\infty$.
\end{enumerate}
Throughout this paper, we always assume that (A1)--(A6) hold.
% In the study of equation \eqref{P} with the optimal transportation, the boundary $\partial\Omega$ needs only to be Lipschitz. When passing to the limit and considering the associated classical PDE, $\partial\Omega$ should be $C^2$.
A weaker lower bound $\varepsilon_n\gg \delta_n$ of the connection radius is required for the study of the graph $p$-Laplaian regularization \cite{slepcev2019analysis}.
As we shall see in Lemma \ref{le:3.2}, assumption (A4) ensures the consistency of the graph Laplacian and the classical Laplacian and is necessary for our problem.
% If we consider piecewise constant kernel $\eta$, it is possible to weaken assumption (A3) to $\varepsilon_n\gg \delta_n^{d/(d+1)}$.
% We also have a stronger smoothness requirement than the graph $p$-Laplaian regularization \cite{slepcev2019analysis} for the density $\rho$ so that we can utilize the $L^p$ theory of elliptic equations.

\subsection{Mathematical tools}

The main idea of studying the continuum limit of the $p$-biharmonic equation on graphs is to rewrite it as a nonlocal equation and establish \textit{a priori} estimates for the nonlocal equation.
The connection between the probability measure $\mu$ and the associated empirical measure $\mu_n=\frac{1}{n}\sum_{i=1}^{n}\delta_{x_i}$ is required.
In fact, there exists a transportation map $T_n$ from $\mu$ to $\mu_n$, denoted by $T_{n \sharp}\mu=\mu_n$. More precisely, $T_{n \sharp}\mu$ is the push-forward measure of $\mu$ by $T_n$ and
\begin{equation*}
  T_{n \sharp}\mu(A)=\mu\left(T^{-1}_n(A)\right),\quad A\in\mathcal{B}(\Omega),
\end{equation*}
where $\mathcal{B}(\Omega)$ is the Borel $\sigma$-algebra.
By a change of variables,
\begin{equation*}
  \frac{1}{n}\sum_{i=1}^{n}\varphi(x_i)=\int_{\Omega}\varphi(x)d\mu_n(x)
  =\int_{\Omega}\varphi(T_n(x))d\mu(x)=\int_{\Omega}\varphi(T_n(x))\rho(x)dx,
\end{equation*}
for any function $\varphi:\Omega_n\rightarrow\mathbb{R}$.

The following result states the convergence and the convergence rate of $T_n$ to the identity operator $Id$ when the graph becomes denser \cite{trillos2015rate}.
\begin{lemma}\label{le:2.1}
  Let $d\geq 3$ and $\{X_i\}_{i=1}^\infty$ be a sequence of independent random variables with distribution $\mu$ on $\Omega$. Then there exists a sequence of transportation maps $\{T_n\}_{n=1}^\infty$ from $\mu$ to the associated empirical measure $\mu_n$, such that
  \begin{align}\label{measure}
  \begin{split}
    c\leq\liminf_{n\rightarrow\infty}\frac{\|Id-T_n\|_{L^\infty(\Omega)}}{\delta_n}
    \leq\limsup_{n\rightarrow\infty}\frac{\|Id-T_n\|_{L^\infty(\Omega)}}{\delta_n}
    \leq C,
  \end{split}
  \end{align}
  almost surely, where $\delta_n$ is given by \eqref{delta}.
\end{lemma}

Assumption (A4) and \eqref{measure} imply
\begin{equation}\label{eq:Tn}
  \lim_{n\rightarrow\infty}\frac{\|Id-T_n\|_{L^\infty(\Omega)}}{\varepsilon_n^2}=0.
\end{equation}
This is what we need in the following.

\begin{remark}
  When $d=2$,
  there exist a sequence of probability measures $\{\tilde{\mu}_n\}_{n=1}^\infty$ (which are
  absolutely continuous with respect to the measure $\mu$) and a sequence of transportation maps $\{\tilde{T}_n\}_{n=1}^\infty$ from $\tilde{\mu}_n$ to $\mu_n$,
  such that $\|Id-\tilde{T}_n\|_{L^\infty(\Omega)}\ll\varepsilon_n^2$
  and $\frac{d\tilde{\mu}_n}{d\mu}\rightarrow 1$ as $n\rightarrow \infty$ uniformly on $\overline{\Omega}$. See \cite[Lemma 3.1]{caroccia2020mumford}.
  In view of this, in the following we are only concerned with the case $d = 3$.
  All the proofs apply, with minor modifications, to the case $d = 2$.
\end{remark}

Lemma \ref{le:2.1} is stated in the random setting. As in \cite{garcia2016continuum}, we regard the data points as the realization of the random variables $\{X_i\}_{i=1}^\infty$ and consider our problem in the deterministic setting.

We shall consider the limit of PDEs defined on graph $G_n$ when $n\rightarrow\infty$. It requires comparing functions defined on $G_n$ and functions defined on $\Omega$.
The $TL^p$ topology, which was first introduced in \cite{garcia2016continuum}, has become an important tool for this purpose.
It is defined as
\begin{equation*}
  TL^p(\Omega)=\{(\nu,g): \nu\in\mathcal{P}(\overline{\Omega}), g\in L^p(\nu)\},
\end{equation*}
where $\mathcal{P}(\overline{\Omega})$ is the space of probability measures. It is a metric space with
\begin{align*}
  d^p_{TL^p}((\nu_1,g_1),(\nu_2,g_2))=\inf_{\pi\in\Pi(\nu_1,\nu_2)}\left\{\iint_{\Omega\times\Omega}
  |x-y|^p+|g_1(x)-g_2(y)|^pd\pi(x,y)\right\},
\end{align*}
where $\Pi(\nu_1,\nu_2)$ is the set of transportation plans.
In our case when $\nu_1$ satisfies assumption (A6), it has an equivalent form
\begin{align*}
  d^p_{TL^p}((\nu_1,g_1),(\nu_2,g_2))=\inf_{T_\sharp\nu_1=\nu_2}\left\{\iint_{\Omega\times\Omega}
  |x-T(x)|^p+|g_1(x)-g_2(T(x))|^pd\nu_1(x)\right\}.
\end{align*}
The $TL^p$ metric is a generalization of the $L^p$ metric with an additional term $|x-T(x)|^p$ that accounts for the transportation distance between measures $\nu_1$ and $\nu_2$.

A simple characterization of the convergence of sequences in $TL^p(\Omega)$ \cite{garcia2016continuum} is given as follows.
\begin{lemma}\label{le:2.2}
  Let $(\nu,g)\in TL^p(\Omega)$ and $\{(\nu_n,g_n)\}_{n=1}^\infty\subset TL^p(\Omega)$, where $\nu$ is absolutely continuous with respective to the Lebesgue measure. Then $(\nu_n,g_n)\rightarrow (\nu,g)$ in $TL^p(\Omega)$ as $n\rightarrow\infty$ if and only if
  $\nu_n\rightarrow\nu$ weakly as $n\rightarrow\infty$ and
  there exists a sequence of transportation maps $\{T_n\}_{n=1}^\infty$ from $\nu$ to $\nu_n$ such that
  \begin{align*}
    \lim_{n\rightarrow\infty}\int_{\Omega}|x-T_n(x)|^pd\nu(x)=0,
  \end{align*}
  and
  \begin{align*}
    \lim_{n\rightarrow\infty}\int_{\Omega}|g(x)-g_n(T_n(x))|^pd\nu(x)=0.
  \end{align*}
\end{lemma}

\begin{remark}\label{re:1}
  Since $\mu$ satisfies assumption (A6) and $\mu_n$ is the associated empirical measure, to conclude the convergence of $(\mu_n,g_n)$ to $(\mu,g)$ in $TL^p(\Omega)$ one only needs to verify the convergence of $g_n\circ T_n$ to $g$ in $L^p(\Omega)$.
\end{remark}

We recall a compactness result for nonlocal functionals
\cite{andreu2008nonlocal}.
\begin{lemma}\label{le:2.3}
  {Let $g$ and $\{g_n\}_{n=1}^\infty$ be functions in $L^p(\Omega)$. If $g_n\rightharpoonup g$ in $L^p(\Omega)$}
  and
  \begin{equation*}
    \frac{1}{\epsilon_n^p}\int_\Omega\int_\Omega\eta_{\epsilon_n}(|x-y|)
    |g_n(y)-g_n(x)|^pdydx\leq C,
  \end{equation*}
  for a sequence $\{\epsilon_n\}_{n=1}^\infty$ converging to $0$ as $n\rightarrow\infty$, then $g\in W^{1,p}(\Omega)$ and
  \begin{align*}
    (\eta(z))^{1/p}\chi_\Omega(x+\epsilon_nz)
    \frac{g_n(x+\epsilon_nz)-g_n(x)}{\epsilon_n}\rightharpoonup
    (\eta(z))^{1/p}z\cdot\nabla g(x), \quad \textrm{in  }L^p(\Omega)\times L^p(\mathbb{R}^d).
  \end{align*}
  Moreover, there exists a subsequence $\{g_{n_k}\}_{k=1}^\infty$, such that
  \begin{equation*}
    g_{n_k}\rightarrow g,\quad \textrm{in  }L^p(\Omega).
  \end{equation*}
\end{lemma}

The nonlocal Poincar\'{e}'s inequality follows from Lemma \ref{le:2.3}.  Inequalities of this type have been studied in \cite{andreu2008nonlocal,radu2017nonlocal}.

\begin{lemma}\label{le:2.4}
  Let $\{g_n\}_{n=1}^\infty$ be functions in $L^p(\Omega)$.
  Then
  \begin{equation}\label{eq:le:2.4:1}
    \int_\Omega|g_n-\overline{g}_n|^pdx\leq \frac{C}{\epsilon_n^p}\int_\Omega\int_\Omega\eta_{\epsilon_n}(|x-y|)
    |g_n(y)-g_n(x)|^pdydx,
  \end{equation}
  holds whenever the right-hand side is bounded. Here $\overline{g}_n=\frac{1}{|\Omega|}\int_\Omega g_ndx$ and the constant $C$ depends only on $p$, $d$, and $\Omega$.
\end{lemma}
\begin{proof}
  It suffices to prove \eqref{eq:le:2.4:1} for $g_n$ with $\overline{g}_n=0$. Otherwise, we replace $g_n$ with $g_n-\overline{g}_n$.

  The proof is by contradiction. Were \eqref{eq:le:2.4:1} false, there would exist a sequence of constants $C_n\rightarrow\infty$ with
  \begin{equation*}
    \int_\Omega|g_n|^pdx> C_n\frac{1}{\epsilon_n^p}\int_\Omega\int_\Omega\eta_{\epsilon_n}(|x-y|)
    |g_n(y)-g_n(x)|^pdydx.
  \end{equation*}
  Let
  \begin{equation*}
    h_n=\frac{g_n}{\|g_n\|_{L^p(\Omega)}}.
  \end{equation*}
  Then
  \begin{equation*}
    \int_\Omega h_ndx=0,\quad \|h_n\|_{L^p(\Omega)}=1,
  \end{equation*}
  and
  \begin{equation*}
    \lim_{n\rightarrow\infty}\frac{1}{\epsilon_n^p}\int_\Omega\int_\Omega\eta_{\epsilon_n}(|x-y|)
    |h_n(y)-h_n(x)|^pdydx=0.
  \end{equation*}
  By the $L^p$ boundedness of $h_n$, we find
  a function $h\in L^p(\Omega)$ such that
  \begin{equation*}
    h_n\rightharpoonup h,\quad \textrm{in } L^p(\Omega).
  \end{equation*}
  We apply Lemma \ref{le:2.3} to see that $h\in W^{1,p}(\Omega)$ and
  \begin{align*}
    (\eta(z))^{1/p}\chi_\Omega(x+\epsilon_nz)
    \frac{h_n(x+\epsilon_nz)-h_n(x)}{\epsilon_n}\rightharpoonup
    (\eta(z))^{1/p}z\cdot\nabla h(x), \quad \textrm{in  }L^p(\Omega)\times L^p(\mathbb{R}^d).
  \end{align*}
By the weak lower semi-continuity of the $L^p$ norm,
\begin{align*}
  \sigma_\eta\int_\Omega|\nabla h|^pdx&= \int_\Omega\int_{\mathbb{R}^d}\eta(z)|z\cdot\nabla h(x)|^pdzdx\\
   &\leq\liminf_{n\rightarrow\infty}\frac{1}{\epsilon_n^p}\int_\Omega\int_\Omega\eta_{\epsilon_n}(|x-y|)
  |h_n(y)-h_n(x)|^pdydx=0,
\end{align*}
where $\sigma_\eta=\int_{\mathbb{R}^d}\eta(z)|z_1|^pdz>0$ and $z_1$ is the first coordinate of $z$.
  The classical Poincar\'{e} inequality implies $\int_\Omega|h|^pdx=0$, which contradicts the fact that $\|h_n\|_{L^p(\Omega)}=1$.
\end{proof}

\section{Consistency for the nonlocal and graph Laplacian}
The graph $p$-biharmonic operator has the graph Laplacian as its ingredient.
An analysis of the consistency for the latter is in order. We start with the nonlocal Laplacian since by the transportation map $T_n$ the graph Laplacian can be rewritten as a nonlocal one.

Let us introduce the weighted nonlocal Laplacian
  \begin{equation*}
    \Delta^\eta_{\rho,\varepsilon}\varphi(x)=\frac{1}{\varepsilon^2}\int_\Omega{\eta}_{\varepsilon}
  (|x-y|)(\varphi(y)-\varphi(x))\rho(y)dy, \quad \varepsilon>0,
  \end{equation*}
and the associated weighted continuous Laplacian
  \begin{equation*}
    \Delta_\rho \varphi =\frac{\sigma_\eta}{2\rho}\textrm{div}(\rho^2\nabla\varphi),
  \end{equation*}
for a function $\varphi:\Omega\rightarrow\mathbb{R}$,  where the constant $\sigma_\eta=\int_{\mathbb{R}^d}\eta(|z|)|z_1|^2dz$.

The following lemma states the consistency of $\Delta^\eta_{\rho,\varepsilon}$ and $\Delta_\rho$ as $\varepsilon\rightarrow 0$ for sufficiently smooth functions with homogeneous Neumann boundary conditions.

\begin{lemma}\label{le:3.1}
  Let $\varphi\in C^2(\mathbb{R}^d)$ with $\frac{\partial\varphi}{\partial\vec{n}}=0$ on $\partial\Omega$.
  Then
  \begin{equation*}
    \lim_{\varepsilon\rightarrow 0}\left\|\Delta^\eta_{\rho,\varepsilon}\varphi-\Delta_\rho \varphi\right\|_{L^\infty(\Omega)}=0.
  \end{equation*}
\end{lemma}

\begin{proof}
  For any $x\in\Omega$,
  \begin{align}\label{eq:le:3.1:1}
    \begin{split}
    \Delta^\eta_{\rho,\varepsilon}\varphi(x)=&\frac{1}{\varepsilon^2}\int_{\mathbb{R}^d}{\eta}_{\varepsilon}
  (|x-y|)(\varphi(y)-\varphi(x))\rho(y)dy\\
  &-\frac{1}{\varepsilon^2}\int_{\mathbb{R}^d\backslash\Omega}{\eta}_{\varepsilon}
  (|x-y|)(\varphi(y)-\varphi(x))\rho(y)dy=:I_1+I_2.
    \end{split}
\end{align}
  By a change of variables $y=x+\varepsilon z$ and Taylor's expansion,
  \begin{align*}
    I_1&=\frac{1}{\varepsilon^2}\int_{\mathbb{R}^d}{\eta}_{\varepsilon}
  (|x-y|)(\varphi(y)-\varphi(x))(\rho(y)-\rho(x))dy \\
  &\quad+\frac{1}{\varepsilon^2}\int_{\mathbb{R}^d}{\eta}_{\varepsilon}
  (|x-y|)(\varphi(y)-\varphi(x))\rho(x)dy\\
  &=\int_{\mathbb{R}^d}\eta(|z|)(\nabla\varphi(x)\cdot z)(\nabla\rho(x)\cdot z)dz
  +\frac{1}{2}\int_{\mathbb{R}^d}\eta(|z|)\sum_{i,j=1}^{d}
  z_iz_j\varphi_{x_ix_j}(x)\rho(x)dz +o(1)\\
  &=:I_3+I_4+o(1).
  \end{align*}
  If $i\neq j$,
    \begin{align*}
      \int_{\mathbb{R}^d}\eta(|z|)\left(\varphi_{x_i}(x)z_i\right)\left(\rho_{x_j}(x)z_j\right) dz =0,
    \end{align*}
  which leads to
  \begin{align*}
    I_3&=\int_{\mathbb{R}^d}\eta(|z|)\left(\sum_{i=1}^d\varphi_{x_i}(x)z_i\right)\left(\sum_{i=1}^d\rho_{x_i}(x)z_i\right)dz
      =\int_{\mathbb{R}^d}\eta(|z|)\sum_{i=1}^d\varphi_{x_i}(x)\rho_{x_i}(x)|z_i|^2dz\\
      &=\sigma_\eta\nabla\varphi(x)\cdot\nabla\rho(x).
  \end{align*}
  Similarly, we have
  \begin{equation*}
    I_4
    =\frac{1}{2}\int_{\mathbb{R}^d}\eta(|z|)\sum_{i=1}^{d}
  |z_i|^2\varphi_{x_ix_i}(x)\rho(x)dz
  =\frac{1}{2}\sigma_\eta\rho(x)\Delta\varphi(x).
  \end{equation*}
  Substituting $I_3$ and $I_4$ into $I_1$ to find
  \begin{equation*}
    I_1=\sigma_\eta\nabla\varphi(x)\cdot\nabla\rho(x)+\frac{1}{2}\sigma_\eta\rho(x)\Delta\varphi(x)+o(1)=\Delta_\rho \varphi(x)+o(1).
  \end{equation*}

  To estimate $I_2$, we recall a lemma from \cite{cortazar2008approximate}, which states that for any $x\in\{z\in\Omega : \textrm{dist}(z,\partial\Omega)\leq 2\varepsilon\}$ and small $\varepsilon$,
  \begin{align*}
    \frac{1}{\varepsilon^2}&\int_{\mathbb{R}^d\backslash\Omega}{\eta}_{\varepsilon}
  (|x-y|)(\varphi(y)-\varphi(x))dy\\
  =&\frac{1}{\varepsilon}\int_{\mathbb{R}^d\backslash\Omega}
  {\eta}_{\varepsilon}(|x-y|)\vec{n}(\bar{x})\cdot\frac{(y-x)}{\varepsilon}
  \frac{\partial\varphi}{\partial\vec{n}}(\bar{x}) dy\\
  &+\int_{\mathbb{R}^d\backslash\Omega}
  {\eta}_{\varepsilon}(|x-y|)\sum_{|\beta|=2}\frac{D^\beta\varphi}{2}(\bar{x})
  \left[\left(\frac{y-\bar{x}}{2}\right)^\beta-
  \left(\frac{x-\bar{x}}{2}\right)^\beta\right]dy+o(1),
  \end{align*}
  where $\bar{x}$ is the orthogonal projection of $x$ on $\partial\Omega$.
  The homogeneous Neumann boundary condition of $\varphi$ implies
  \begin{align*}
    |I_2|\leq C\varepsilon^2\int_{\mathbb{R}^d\backslash\Omega}
  {\eta}_{\varepsilon}(|x-y|)dy+o(1).
  \end{align*}
  The proof is finished by substituting it into \eqref{eq:le:3.1:1} and passing to the limit $\varepsilon\rightarrow 0$.
\end{proof}

In the following, we generalize Lemma \ref{le:3.1} to the graph case.
Let $T_n$ be the transportation map given in Lemma \ref{le:2.1}.
 Denote
\begin{equation*}
  \tilde{\varphi}(x)=\varphi(T_n(x)),\quad x\in\Omega,
\end{equation*}
for $\varphi:\Omega_n\rightarrow\mathbb{R}$, and
\begin{equation*}
  \tilde{\eta}_{\varepsilon_n}(x,y)={\eta}_{\varepsilon_n}
  (|T_n(x)-T_n(y)|),\quad x,y\in\Omega.
\end{equation*}
Now the graph Laplacian of $\varphi$ can be rewritten as a nonlocal Laplacian.
Namely,
\begin{align*}
  \Delta_{G_n}\varphi(T_n(x))&=\frac{1}{n\varepsilon_n^2}\sum_{j=1}^n\eta_{\varepsilon_n}
  (|T_n(x)-x_j|)(\varphi(x_j)-\varphi(T_n(x)))\\
  &=\frac{1}{\varepsilon_n^2}\int_\Omega\tilde{\eta}_{\varepsilon_n}
  (x,y)(\tilde{\varphi}(y)-\tilde{\varphi}(x))\rho(y)dy\\
  &=\Delta^{\tilde{\eta}}_{\rho,\varepsilon_n}\tilde{\varphi}(x),
\end{align*}
for any $x\in\Omega$.
% A stronger assumption on the lower bound of the connection radius $\varepsilon_n$ is needed.
\begin{lemma}\label{le:3.2}
Let $\varphi\in C^2(\mathbb{R}^d)$ with $\frac{\partial\varphi}{\partial\vec{n}}=0$ on $\partial\Omega$.
% If
% \begin{equation}\label{eq:le:3.00}
%   \lim_{n\rightarrow\infty}\frac{\|Id-T_n\|_{L^\infty(\Omega)}}{\varepsilon_n^2}= 0,
% \end{equation}
Then
\begin{equation*}
  \lim_{n\rightarrow\infty}\|
  \Delta^{\tilde{\eta}}_{\rho,\varepsilon_n}\tilde{\varphi}
  -\Delta_\rho \varphi\|_{L^\infty(\Omega)}=0.
\end{equation*}
\end{lemma}

\begin{proof}
Notice that
\begin{align*}
  \|\Delta^{\tilde{\eta}}_{\rho,\varepsilon_n}\tilde{\varphi}-\Delta_\rho \varphi\|_{L^\infty(\Omega)}
  \leq& \left\|\Delta^{\tilde{\eta}}_{\rho,\varepsilon_n}\tilde{\varphi}-\Delta^{\tilde{\eta}}_{\rho,\varepsilon_n}\varphi\right\|_{L^\infty(\Omega)}\\
  &+\left\|\Delta^{\tilde{\eta}}_{\rho,\varepsilon_n}{\varphi}-\Delta^\eta_{\rho,\varepsilon_n}\varphi\right\|_{L^\infty(\Omega)}
  +\left\|\Delta^\eta_{\rho,\varepsilon_n}\varphi-\Delta_\rho \varphi\right\|_{L^\infty(\Omega)}\\
  =:& I_1+I_2+I_3.
\end{align*}
By Lemma \ref{le:3.1}
we only need to show that both $I_1$ and $I_2$ go to zero as $n\rightarrow \infty$.

For any $x,y\in\Omega$,
\begin{equation*}
  \frac{|T_n(x)-T_n(y)|}{\varepsilon_n}\leq 2\Longrightarrow
  \frac{|x-y|}{\tilde{\varepsilon}_n}\leq 2,
\end{equation*}
where $\tilde{\varepsilon}_n=\varepsilon_n+\|Id-T_n\|_{L^\infty(\Omega)}$.
Assumption (A3) and \eqref{eq:Tn} yield
\begin{equation*}
  \int_\Omega\tilde{\eta}_{\varepsilon_n}(x,y)dy
  \leq C\left(\frac{\tilde{\varepsilon}_n}{\varepsilon_n}\right)^d\leq C,
\end{equation*}
for any $x\in\Omega$ and large $n$.
It follows from \eqref{eq:Tn} that
\begin{align*}
  |\Delta^{\tilde{\eta}}_{\rho,\varepsilon_n}&\tilde{\varphi}(x)-\Delta^{\tilde{\eta}}_{\rho,\varepsilon_n}\varphi(x)|\\
  =&\left|\frac{1}{\varepsilon_n^2}\int_\Omega
  \tilde{\eta}_{\varepsilon_n}(x,y)
  (\varphi(T_n(y))-\varphi(T_n(x)))-((\varphi(y)-\varphi(x)))
  \rho(y)dy\right|\\
  \leq& \frac{C}{\varepsilon_n^{2}}
  \left(\esssup_{y\in\Omega}|T_n(y)-y|+\esssup_{x\in\Omega}|T_n(x)-x|\right)  \\
  \leq & \frac{C\|Id-T_n\|_{L^\infty(\Omega)}}{\varepsilon_n^2},
\end{align*}
and hence $I_1\rightarrow 0$.

To estimate $I_2$, we first assume that $\eta$ is Lipschitz. Then
\begin{align}\label{eq:le:3.2:1}
  \begin{split}
  |\Delta^{\tilde{\eta}}_{\rho,\varepsilon_n}&{\varphi}(x)-\Delta^\eta_{\rho,\varepsilon_n}\varphi(x)|\\
  =&\left|\frac{1}{\varepsilon_n^2}\int_\Omega
  \left(\tilde{\eta}_{\varepsilon_n}(x,y)-{\eta}_{\varepsilon_n}(|x-y|)\right)
  (\varphi(y)-\varphi(x))\rho(y)dy  \right|\\
  \leq& \frac{C}{\varepsilon_n^{d+2}}\int_{\Omega\cap B(x,2\tilde{\varepsilon}_n)}
  \left|\frac{|T_n(x)-T_n(y)|}{\varepsilon_n}-\frac{|x-y|}{\varepsilon_n}\right|
  |\varphi(y)-\varphi(x)|dy\\
  \leq & \frac{C\|Id-T_n\|_{L^\infty(\Omega)}}{\varepsilon_n^2},
  \end{split}
\end{align}
from which we have $I_2\rightarrow 0$.

Now we assume that $\eta$ has only one discontinuity point $\alpha > 0$. Let $x\in\Omega$ and
\begin{equation*}
  \hat{\varepsilon}_n=\varepsilon_n-\frac{2}{\alpha}\|Id-T_n\|_{L^\infty(\Omega)},\quad
  \tilde{\varepsilon}_n=\varepsilon_n+\frac{2}{\alpha}\|Id-T_n\|_{L^\infty(\Omega)}.
\end{equation*}
Then
\begin{align}\label{eq:le:3.2:2}
  \begin{split}
  |x-y|< \alpha\hat{\varepsilon}_n \Longrightarrow \frac{|x-y|}{\varepsilon_n}<\alpha, ~\text{ and }~ \frac{|T_n(x)-T_n(y)|}{\varepsilon_n}<\alpha, \\
  |x-y|> \alpha\tilde{\varepsilon}_n \Longrightarrow \frac{|x-y|}{\varepsilon_n}>\alpha, ~\text{ and }~ \frac{|T_n(x)-T_n(y)|}{\varepsilon_n}>\alpha,
  \end{split}
\end{align}
We repeat \eqref{eq:le:3.2:1} and split the integral in the second line into three parts, i.e.,
\begin{equation*}
  \int_\Omega= \int_{|x-y|< \alpha\hat{\varepsilon}_n}
  +\int_{|x-y|> \alpha\tilde{\varepsilon}_n}
  +\int_{\alpha\hat{\varepsilon}_n\leq|x-y|\leq \alpha\tilde{\varepsilon}_n}.
\end{equation*}
Owning to \eqref{eq:le:3.2:2} and the Lipschitz continuous of $\eta$ on $[0,\alpha)$ and $(\alpha,\infty)$, we can handle the first and the
second integral on the right-hand side in the same way as \eqref{eq:le:3.2:1}. For the third integral, we have
\begin{align*}
  \left|\frac{1}{\varepsilon_n^2}\int_{\alpha\hat{\varepsilon}_n\leq|x-y|\leq \alpha\tilde{\varepsilon}_n}
  \left(\tilde{\eta}_{\varepsilon_n}(x,y)-{\eta}_{\varepsilon_n}(|x-y|)\right)
  (\varphi(y)-\varphi(x))\rho(y)dy  \right| \\
  \leq \frac{C}{\varepsilon^{d+1}_n}\int_{\alpha\hat{\varepsilon}_n\leq|x-y|\leq \alpha\tilde{\varepsilon}_n}dy
  \leq \frac{C(\tilde{\varepsilon}_n^{d}-\hat{\varepsilon}_n^{d})}{\varepsilon^{d+1}_n}\leq\frac{C\|Id-T_n\|_{L^\infty(\Omega)}}{\varepsilon^{2}_n}.
\end{align*}
Then $I_2\rightarrow 0$ and the conclusion holds. Clearly, the above proof works for finite many
discontinuity points, i.e., for piecewise Lipschitz $\eta$. This completes the proof.
\end{proof}

% \begin{remark}\label{remark:3.3}
% From the proof of Lemma \ref{le:3.2}, we have
% \begin{equation}\label{eq:remark3.3:1}
%   \lim_{n\rightarrow\infty}\|
%   \Delta^{\tilde{\eta}}_{\rho,\varepsilon_n}{\varphi}
%   -\Delta_\rho \varphi\|_{L^\infty(\Omega)}=0,
% \end{equation}
% for any $\varphi\in C^2(\mathbb{R}^d)$ with $\frac{\partial\varphi}{\partial\vec{n}}=0$ on $\partial\Omega$.
% If further $\eta$ is piecewise constant, \eqref{eq:remark3.3:1} holds under a weaker condition on $\varepsilon_n$ than \eqref{eq:Tn},
% \begin{equation*}
%   \lim_{n\rightarrow\infty}\frac{\|Id-T_n\|^d_{L^\infty(\Omega)}}{\varepsilon^{d+1}_n}=0.
% \end{equation*}
% \end{remark}

\begin{remark}
We used \eqref{eq:Tn} in the proof of Lemma \ref{le:3.2} to
control the error caused by the transportation map.
The necessity of it can be illustrated by the following example.

Let $\Omega=[-\frac{1}{2}, \frac{1}{2}]$ and $\Omega_n=\{\cdots, -\frac{3}{2n}, -\frac{1}{2n}, 0, \frac{1}{n}, \frac{2}{n},\cdots\}$. We have the bound on the transportation distance
\begin{equation}\label{eq:re:1:1}
  \limsup_{n\rightarrow\infty}\frac{\|Id-T_n\|_{L^\infty(\Omega)}}{1/n}\leq C.
\end{equation}
Let $\eta(s)=\chi_{[0,1]}(s)$, $u(x)=x$, and $x_i=0$. Then
\begin{align*}
  \Delta_{G_n}u(x_i)=\frac{1}{n\varepsilon_n^3}\sum_{\{j: |x_j|\leq \varepsilon_n\}}x_j\approx \frac{1}{n\varepsilon_n^3}\frac{ n\varepsilon_n}{2n}=\frac{1}{2n\varepsilon_n^2}
\end{align*}
converges to $0=\Delta u(x_i)$ as $n\rightarrow\infty$ if and only if $\frac{1}{n\varepsilon_n^2}\rightarrow 0$. This and \eqref{eq:re:1:1} imply \eqref{eq:Tn}.
\end{remark}

\section {\textit{A priori} estimates for nonlocal and graph Poisson equations}

It is well-known that the classical Poisson equation admits a $W^{2,p}$ solution for a given $L^p$ datum.
We cannot expect the same property for nonlocal Poisson equations, since
the lack of regularizing effect \cite{chasseigne2006asymptotic,radu2017nonlocal}.
Alternatively, we present here \textit{a priori} estimates of the solution and its nonlocal gradient.

\begin{theorem}\label{th:3.3}
 Let $h\in L^q(\Omega)$, $1<q\leq \infty$. If $u\in L^1(\Omega)$ is a solution of the nonlocal Poisson equation
  \begin{equation*}
    -\Delta^\eta_{\rho,\varepsilon}u(x)=h(x),\quad x\in\Omega,
  \end{equation*}
then $u\in L^q(\Omega)$. Furthermore,
\begin{equation}\label{eq:th:3.3:0a}
  \int_\Omega |u|^qdx+\int_{\Omega}\int_{\Omega}\eta_\varepsilon(|x-y|)\left|\frac{u(y)-u(x)}{\varepsilon}\right|^qdydx
  \leq C\left(\int_\Omega|u|dx\right)^q
  +C\int_\Omega|h|^qdx+C,
\end{equation}
for $1<q< 2$,
\begin{align}\label{eq:th:3.3:0b}
  \begin{split}
  \int_\Omega |u|^qdx+\int_\Omega\int_\Omega\eta_{\varepsilon}(|x-y|)\left|\frac{(|u|^{(q-2)/2}u)(y)-(|u|^{(q-2)/2}u)(x)}{\varepsilon}\right|^2dydx \qquad\\
  \leq C\left(\int_\Omega|u^{q/2}|dx\right)^2 + C\int_\Omega|h|^qdx,
  \end{split}
\end{align}
for $q\geq2$, and
\begin{equation}\label{eq:th:3.3:0c}
  \|u\|_{L^\infty(\Omega)}\leq C \|u\|_{L^1(\Omega)} + C\|h\|_{L^\infty(\Omega)},
\end{equation}
for $q=\infty$.
Here the constant $C$ in \eqref{eq:th:3.3:0a}--\eqref{eq:th:3.3:0c} does not depend on $\varepsilon$.
\end{theorem}
Since the weight $\rho$ has positive lower and upper bounds, we can replace the Lebesgue measure with $d\mu(x)=\rho(x)dx$ in the proof and simply consider the unweighted nonlocal Laplacian $\Delta^\eta_{\varepsilon}:=\Delta^\eta_{\rho\equiv 1,\varepsilon}$.

\begin{proof}
1.
We rewrite the nonlocal Poisson equation as
\begin{equation*}
  \frac{1}{\varepsilon^2}\int_{\Omega}\eta_\varepsilon(|x-y|)dy u(x)
  =\frac{1}{\varepsilon^2}\int_{\Omega}\eta_\varepsilon(|x-y|)u(y)dy+h(x),\quad x\in\Omega.
\end{equation*}
The conclusion $u\in L^q(\Omega)$ follows from $u\in L^1(\Omega)$ and $h\in L^q(\Omega)$.

  Let $2|u|^{q-2}u\in L^{q/(q-1)}(\Omega)$ be the test function for the nonlocal Poisson equation.
  % We use the notation $\frac{0}{0}=0$ for $1<q<2$.
  By  nonlocal integration by parts and Young's inequality,
  \begin{align}\label{eq:th:3.3:1}
    \begin{split}
      &\frac{1}{\varepsilon^2}\int_\Omega\int_\Omega\eta_{\varepsilon}(|x-y|)(u(y)
    -u(x))\left((|u|^{q-2}u)(y)
    -(|u|^{q-2}u)(x)\right)dydx\\
    &=-2\int_\Omega\Delta^{\eta}_{\varepsilon}u
   \cdot |u|^{q-2}udx=2\int_\Omega h|u|^{q-2}udx
   \leq \delta\int_{\Omega}|u|^qdx+
   \frac{C}{\delta^{q-1} }\int_\Omega|h|^qdx,
    \end{split}
  \end{align}
  where $\delta>0$ is a constant.

In the following, we consider the cases $1<q< 2$, $q\geq2$, and $q=\infty$ separately.
It relies on inequalities involving $|s|^{q-2}s$.

2. Let $1<q<2$. We recall the following inequality \cite[Section 10, VII]{lindqvist2017notes}
\begin{equation*}
  (b-a)(|b|^{q-2}b-|a|^{q-2}a)\geq (q-1)(b-a)^2\left(1+|a|^2+|b|^2\right)^{(q-2)/2}, \quad 1<q\leq 2,
\end{equation*}
for any $a,b\in\mathbb{R}$. In both here and \eqref{eq:th:3.3:1} we use the notation $|0|^{q-2}0=0$.
By reverse Young's inequality and nonlocal Poincar\'{e}'s inequality \eqref{eq:le:2.4:1}, the left-hand side of \eqref{eq:th:3.3:1}
\begin{align*}
&\geq \frac{q-1}{\varepsilon^2}\int_\Omega\int_\Omega \eta_{\varepsilon}(|x-y|)(u(y)-u(x))^2\left(1+|u(x)|^2+|u(y)|^2\right)^{(q-2)/2}dydx\\
&\geq {C\delta^{\frac{2-q}{2}}}\int_\Omega\int_\Omega \eta_{\varepsilon}(|x-y|)
\left|\frac{u(y)-u(x)}{\varepsilon}\right|^qdydx\\
&\qquad\qquad\qquad\qquad\qquad-C\delta\int_\Omega\int_\Omega \eta_{\varepsilon}(|x-y|)\left(1+|u(x)|^2+|u(y)|^2\right)^{q/2}dydx \\
&\geq {C\delta^{\frac{2-q}{2}}}\int_\Omega |u-\overline{u}|^qdx-C\delta\int_\Omega|u|^qdx-C\delta \\
&\geq {C\delta^{\frac{2-q}{2}}}\int_\Omega |u|^qdx-{C\delta^{\frac{2-q}{2}}}\left(\int_\Omega|u|dx\right)^q-C\delta\int_\Omega|u|^qdx-C\delta,
\end{align*}
Substituting it into \eqref{eq:th:3.3:1} and choosing a sufficiently small $\delta>0$ such that $C\delta^{\frac{2-q}{2}}> C\delta$ we obtain the $L^q$ boundedness of $u$. The estimate of the nonlocal gradient in \eqref{eq:th:3.3:0a} is contained in the above proof.

3. Assume that $q\geq 2$.
It follows from the inequality \cite[Section 10, V]{lindqvist2017notes}
  \begin{equation*}
    (b-a)(|b|^{q-2}b-|a|^{q-2}a)\geq \frac{4}{q^2}\left||b|^{(q-2)/2}b-|a|^{(q-2)/2}a\right|^2, \quad a,b\in\mathbb{R}, \quad q\geq 2,
  \end{equation*}
  and nonlocal Poincar\'{e}'s inequality \eqref{eq:le:2.4:1} that
  the left-hand side of \eqref{eq:th:3.3:1}
  \begin{align*}
    &\geq \frac{4}{q^2\varepsilon^2}\int_\Omega\int_\Omega\eta_{\varepsilon}(|x-y|)\left|(|u|^{(q-2)/2}u)(y)-(|u|^{(q-2)/2}u)(x)\right|^2dydx    \\
    &\geq C\int_\Omega \left||u|^{(q-2)/2}u-\overline{|u|^{(q-2)/2}u}\right|^2dx
    \geq C\int_\Omega|u|^qdx-C\left(\int_\Omega|u^{q/2}|dx\right)^2.
  \end{align*}
 Again, \eqref{eq:th:3.3:0b} is derived from the above estimate and \eqref{eq:th:3.3:1} with  a small $\delta>0$.

 4.
 If $h\equiv 0$, \eqref{eq:th:3.3:0c} follows from \eqref{eq:th:3.3:0b} directly.
 In fact,
 by choosing $q=2^{k+1}$ in \eqref{eq:th:3.3:0b} and neglecting the nonlocal gradient term, we have
\begin{equation*}
  \left(\int_\Omega |{u}|^{2^{k+1}}dx\right)^{\frac{1}{2^{k+1}}}
  \leq C^{\frac{1}{2^{k}}}
  \left(\int_\Omega |{u}|^{2^{k}}dx\right)^{\frac{1}{2^{k}}}.
\end{equation*}
Iterating repeatedly leads to
\begin{equation*}
  \left(\int_\Omega |{u}|^{2^{k+1}}dx\right)^{\frac{1}{2^{k+1}}}
  \leq C^\alpha
  \int_\Omega |{u}|dx,
\end{equation*}
where $\alpha=\sum_{k=0}^{\infty}\frac{1}{2^k}=2$.
Then \eqref{eq:th:3.3:0c} is obtained by letting  $k\rightarrow\infty$.

5. For a general $h$, we rewrite the nonlocal Poisson equation such that we can use the conclusion of step 4.

Let $\phi$ be a classical solution of
 \begin{align}\label{eq:th:3.3:2}
  \begin{cases}
    \Delta \phi=2, \quad & \textrm{in }\Omega,\\
    \frac{\partial \phi}{\partial\vec{n}}=0,\quad & \textrm{on }\partial\Omega.
  \end{cases}
 \end{align}
Then $\phi$ is uniformly bounded in $\Omega$ and by
Lemma \ref{le:3.1},
\begin{equation*}
  \lim_{\varepsilon\rightarrow 0}\Delta^\eta_{\varepsilon}\phi(x) = \Delta \phi(x)=2,
\end{equation*}
for any $x\in\Omega$.

Define $w=u+\|h\|_{L^\infty(\Omega)} \phi$, we have
\begin{equation*}
  -\Delta^\eta_{\varepsilon}w=-\Delta^\eta_{\varepsilon}u
  -\|h\|_{L^\infty(\Omega)}\Delta^\eta_{\varepsilon}\phi
  =h-\|h\|_{L^\infty(\Omega)}\Delta^\eta_{\varepsilon}\phi\leq 0,
\end{equation*}
for a small $\varepsilon$.
Namely, $w$ is a subsolution of the nonlocal Laplace equation.
For $w_+=\max\{0, w\}$, if $w(x)\geq 0$,
\begin{align*}
  \frac{1}{\varepsilon^2}\int_\Omega\eta_{\varepsilon}(|x-y|)w_+(y)dy
  &\geq \frac{1}{\varepsilon^2}\int_\Omega\eta_{\varepsilon}(|x-y|)w(y)dy \\
  &\geq\frac{1}{\varepsilon^2}\int_\Omega\eta_{\varepsilon}(|x-y|)w(x)dy
  =\frac{1}{\varepsilon^2}\int_\Omega\eta_{\varepsilon}(|x-y|)w_+(x)dy.
\end{align*}
On the other hand, if $w(x)< 0$,
\begin{align*}
  \frac{1}{\varepsilon^2}\int_\Omega\eta_{\varepsilon}(|x-y|)w_+(y)dy
  \geq 0
  =\frac{1}{\varepsilon^2}\int_\Omega\eta_{\varepsilon}(|x-y|)w_+(x)dy.
\end{align*}
This means that
\begin{equation*}
  -\Delta^\eta_{\varepsilon}w_+(x)
  =-\frac{1}{\varepsilon^2}\int_\Omega\eta_{\varepsilon}(|x-y|)(w_+(y)-w_+(x))dy\leq 0,
\end{equation*}
for any $x\in\Omega$ and $w_+$ is also a subsolution.

Now we multiply both sides of the above inequality by $2w_+^{q-1}=2|w_+|^{q-2}w_+$ and
integrate over $\Omega$ to obtain an inequality similar to \eqref{eq:th:3.3:1}, i.e.,
\begin{align*}
  &\frac{1}{\varepsilon^2}\int_\Omega\int_\Omega\eta_{\varepsilon}(|x-y|)(w_+(y)
-w_+(x))\left((|w_+|^{q-2}w_+)(y)
-(|w_+|^{q-2}w_+)(x)\right)dydx\\
&=-2\int_\Omega\Delta^{\eta}_{\varepsilon}w_+
\cdot |w_+|^{q-2}w_+dx\leq 0.
\end{align*}
Repeating step 3-4 for it leads to
\begin{equation*}
  \esssup_{x\in\Omega} w_+
  \leq  C\int_\Omega|w_+|dx,
\end{equation*}
which implies
\begin{equation*}
  \esssup_{x\in\Omega} u
  \leq C \int_\Omega |u|dx + C\|h\|_{L^\infty(\Omega)}.
\end{equation*}

To estimate the lower bound of $u$, we define $w=u-\|h\|_{L^\infty(\Omega)} \phi$ and $w_-=\min\{0, w\}$. Then $w_-$ is a supersolution of the nonlocal Laplace equation.
Choosing $2w_-^{q-1}=2|w_-|^{q-2}w_-$ as the test function and repeating step 3-4
to find
\begin{equation*}
  \esssup_{x\in\Omega} -u
  \leq C \int_\Omega |u|dx + C\|h\|_{L^\infty(\Omega)}.
\end{equation*}
This completes the proof of \eqref{eq:th:3.3:0c}.
\end{proof}

% The graph version of Theorem \ref{th:3.3} is omitted since it is simply a restatement with the notation $T_n$.

Similar results hold for the graph Poisson equation.

\begin{corollary}\label{th:3.4}
  Let $u_n$ be a solution of the graph Poisson equation
  \begin{equation*}
    -\Delta_{G_n}u_n(x_i)=h_n(x_i),\quad x_i\in\Omega_n.
  \end{equation*}
Then there exists a nonincreasing and smooth kernel $\zeta: [0,\infty)\rightarrow[0,\infty)$ with compact support, such that
\begin{align}\label{eq:th:3.4:1}
  \begin{split}
  \int_\Omega |\tilde{u}_n|^qdx+\int_{\Omega}\int_{\Omega}\zeta_{\varepsilon_n}(|x-y|)\left|\frac{\tilde{u}_n(y)-\tilde{u}_n(x)}{\varepsilon_n}\right|^qdydx
  \leq C\left(\int_\Omega|\tilde{u}_n|dx\right)^q \\
  +C\int_\Omega|\tilde{h}_n|^qdx+C,
  \end{split}
\end{align}
for $1<q< 2$,
\begin{align}\label{eq:th:3.4:2}
  \begin{split}
  \int_\Omega |\tilde{u}_n|^qdx
  +\int_\Omega\int_\Omega\zeta_{\varepsilon_n}(|x-y|)\left|\frac{(|\tilde{u}_n|^{(q-2)/2}\tilde{u}_n)(y)-(|\tilde{u}_n|^{(q-2)/2}\tilde{u}_n)(x)}{\varepsilon_n}\right|^2dydx \\
  \leq C\left(\int_\Omega|\tilde{u}_n^{q/2}|dx\right)^2 +C\int_\Omega|\tilde{h}_n|^qdx,
  \end{split}
\end{align}
for $q\geq2$, and
\begin{equation}\label{eq:th:3.4:2a}
  \|\tilde{u}_n\|_{L^\infty(\Omega)}\leq C \|\tilde{u}_n\|_{L^1(\Omega)} + C\|\tilde{h}_n\|_{L^\infty(\Omega)},
\end{equation}
for $q=\infty$.
Here $\zeta_{\varepsilon_n}(s)=\frac{1}{\varepsilon_n^d}\zeta(s/\varepsilon_n)$, $\tilde{u}_n=u\circ T_n$, and
the constant $C$ in \eqref{eq:th:3.4:1}--\eqref{eq:th:3.4:2a} does not depend on $n$ and $\varepsilon_n$.
\end{corollary}

% We leave the estimate for $q=\infty$ to the next section where a more general situation is considered.

\begin{proof}
  By assumption (A3), we have $\eta(1)\geq 1$ and $\eta(2)=0$.
  Let $\zeta: [0,\infty)\rightarrow[0,\infty)$ be a nonincreasing and smooth function with $\zeta(0)< 1$ and $\zeta(\frac{1}{2})=0$.
  For any $x,y\in\Omega$,
  \begin{equation*}
    \frac{|x-y|}{\varepsilon_n}\leq \frac{1}{2}
    \Longrightarrow
    \frac{|T_n(x)-T_n(y)|}{\varepsilon_n}
    \leq \frac{|x-y|+2\|Id-T_n\|_{L^\infty(\Omega)}}{\varepsilon_n}\leq  1,
  \end{equation*}
  for large $n$.
  Consequently,
\begin{equation*}%\label{eq:th:3.4:3}
  \zeta_{\varepsilon_n}(|x-y|)=\frac{1}{\varepsilon_n^d}\zeta\left(\frac{|x-y|}{\varepsilon_n}\right)\leq \frac{1}{\varepsilon_n^d}\eta\left(\frac{|T_n(x)-T_n(y)|}{\varepsilon_n}\right)=:\tilde{\eta}_{\varepsilon_n}(x,y), \quad x,y\in\Omega.
\end{equation*}
Let $2|u_n(x_i)|^{q-2}u_n(x_i)$ be the test function for the graph Poisson equation, namely
\begin{equation*}
  -2\frac{1}{n}\sum_{i=1}^{n}\Delta_{G_n}u_n(x_i)|u_n(x_i)|^{q-2}u_n(x_i)
  =2\frac{1}{n}\sum_{i=1}^{n}h_n(x_i)|u_n(x_i)|^{q-2}u_n(x_i).
\end{equation*}
 Then the left-hand side
\begin{align*}
  &=-2\int_\Omega\Delta^{\tilde{\eta}}_{\rho,\varepsilon_n}\tilde{u}_n \cdot |\tilde{u}_n|^{q-2}\tilde{u}_n \rho dx \\
  &=\frac{1}{\varepsilon_n^2}\int_\Omega\int_\Omega\tilde{\eta}_{\varepsilon_n}(x,y)(\tilde{u}_n(y)-\tilde{u}_n(x))\\
  &\qquad\qquad\qquad\qquad\cdot\left((|\tilde{u}_n|^{q-2}\tilde{u}_n)(y)
  -(|\tilde{u}_n|^{q-2}\tilde{u}_n)(x)\right)\rho(y)\rho(x)dydx\\
  &\geq
  \frac{1}{\varepsilon_n^2}\int_\Omega\int_\Omega\zeta_{\varepsilon_n}(|x-y|)(\tilde{u}_n(y)-\tilde{u}_n(x))\\
  &\qquad\qquad\qquad\qquad\cdot\left((|\tilde{u}_n|^{q-2}\tilde{u}_n)(y)
  -(|\tilde{u}_n|^{q-2}\tilde{u}_n)(x)\right)\rho(y)\rho(x)dydx,
\end{align*}
and the right-hand side
\begin{equation*}
  =2\int_\Omega \tilde{h}_n|\tilde{u}_n|^{q-2}\tilde{u}_n\rho dx.
\end{equation*}
Combining both we arrive at an inequality similar to \eqref{eq:th:3.3:1}.
The conclusions are obtained by  repeating the proof of Theorem \ref{th:3.3}.
\end{proof}

\section{Discrete to continuum convergence}
In this section, we combine the results established in Section 3 and Section 4 to study the continuum limit of Equation \eqref{P}.
Before that, we state the unique solvability of it on a fixed graph $G_n$.

\begin{lemma}\label{th:4.1}
  Let $n> 0$ and $\varepsilon_n>0$ be fixed.
  Equation \eqref{P} admits a unique solution $u_n$ with
  \begin{equation}\label{eq:th3.3}
  \frac{1}{n}\sum_{i=1}^n|\Delta_{G_n}u_n(x_i)|^p
  +\frac{\lambda}{2n}\sum_{i=1}^n|u_n(x_i)|^2\leq C.
\end{equation}
If $\lambda>0$ is small, we further have
\begin{equation}\label{eq:th3.3a}
  \max_{x_i\in\Omega_n}|u_n(x_i)|+\max_{x_i\in\Omega_n}|\Delta_{G_n}u_n(x_i)|
  \leq C.
\end{equation}
Here both constants on the right-hand side do not depend on $n$ and $\varepsilon_n$.
\end{lemma}

The existence of a unique minimizer for the variational model associated to equation \eqref{P} has been shown in \cite{el2020discrete}. We give a  sketched proof here.
\begin{proof}

1.
  Let us consider the functional
  \begin{align*}
  E_n(u)=\frac{1}{p}\sum_{i=1}^n|\Delta_{G_n}u(x_i)|^p
  +\frac{\lambda}{2}\sum_{i=1}^n|u(x_i)-f(x_i)|^2,
  \end{align*}
  and introduce the vector space
    \begin{equation*}
    V:=\left\{v: \Omega_n\rightarrow \mathbb{R}\right\}
  \end{equation*}
  equipped with the $\ell^p$-norm
  \begin{equation*}
    \|v\|=\left(\sum_{i=1}^n|v(x_i)|^p\right)^{1/p}.
  \end{equation*}
  It is straightforward to verify that $E_n(u)$ is coercive, lower semi-continuous, and strictly convex on $V$
  and admits a unique minimizer $u_n$, which is also a solution of equation \eqref{P}. The uniqueness of the solution follows from the monotonicity of $|s|^{p-2}s$.

  2.
To obtain the uniform estimate \eqref{eq:th3.3}, we multiply both sides of equation \eqref{P} by $u_n(x_i)/n$
 and sum over $i=1,\cdots,n$ to find
\begin{align*}
  \frac{1}{n}\sum_{i=1}^n
  \Delta_{G_n}(|\Delta_{G_n}u_{n}|^{p-2}\Delta_{G_n}u_{n})(x_i)
  &\cdot u_n(x_i)\\
  &+\frac{\lambda}{n}\sum_{i=1}^n(u_n(x_i)-f(x_i))\cdot u_n(x_i)=0.
\end{align*}
Since for any functions $u, \varphi: \Omega_n\rightarrow\mathbb{R}$,
\begin{equation*}
  \sum_{i=1}^n\Delta_{G_n}u(x_i)\varphi(x_i)=
  \sum_{i=1}^nu(x_i)\Delta_{G_n}\varphi(x_i),
\end{equation*}
we see from Young's inequality that
\begin{align*}
  \frac{1}{n}\sum_{i=1}^n|\Delta_{G_n}u_{n}(x_i)|^p
  +\frac{\lambda}{2n}\sum_{i=1}^n|u_{n}(x_i)|^2
  \leq \frac{\lambda}{2n}\sum_{i=1}^n|f(x_i)|^2.
\end{align*}

3. To prove \eqref{eq:th3.3a},
we introduce
\begin{equation*}
  v_n(x_i)=|\Delta_{G_n}u_{n}(x_i)|^{p-2}\Delta_{G_n}u_{n}(x_i),
  \quad x_i\in\Omega_n,
\end{equation*}
and rewrite equation \eqref{P} as
\begin{equation*}
  -\Delta_{G_n}v_{n}(x_i) + \lambda(f(x_i)-u_n(x_i))=0,\quad x_i\in\Omega_n.
\end{equation*}
By the virtue of \eqref{eq:th:3.4:2a},
\begin{equation*}
  \max_{x_i\in\Omega_n}|\Delta_{G_n}u_{n}(x_i)|^{p-1}=\max_{x_i\in\Omega_n}|v_n(x_i)|\leq
  C\frac{1}{n}\sum_{i=1}^{n}|v_n(x_i)|
  +C\lambda
  \max_{x_i\in\Omega_n}|f(x_i)-u_n(x_i)|.
\end{equation*}
We utilize \eqref{eq:th:3.4:2a} again to obtain
\begin{align*}
  &\max_{x_i\in\Omega_n}|u_n(x_i)|^{p-1}
  \leq C\left(\frac{1}{n}\sum_{i=1}^{n}|u_n(x_i)|\right)^{p-1}
  +C\max_{x_i\in\Omega_n}|\Delta_{G_n}u_n(x_i)|^{p-1}\\
  &\leq C\left(\frac{1}{n}\sum_{i=1}^{n}|u_n(x_i)|\right)^{p-1}
  + \frac{C}{n}\sum_{i=1}^{n}|\Delta_{G_n}u_n(x_i)|^{p-1}+C\lambda\max_{x_i\in\Omega_n}|f(x_i)-u_n(x_i)|.
\end{align*}
Estimate \eqref{eq:th3.3} indicates the boundedness of the first term and the second term on the right-hand side.
If $\lambda>0$ is small such that $$\max_{x_i\in\Omega_n}|u_n(x_i)|^{p-1}>C\lambda \max_{x_i\in\Omega_n}|u_n(x_i)|,$$
 we have
the conclusion \eqref{eq:th3.3a}.
\end{proof}

Let us now discuss the asymptotic behavior for the solution of equation \eqref{P} as $n\rightarrow\infty$.
We shall see that the solution $u_n$ converges to a weak solution of the  weighted $p$-biharmonic equation
\begin{align}\label{CP}
  \begin{cases}
    -\Delta_\rho(|\Delta_\rho u|^{p-2}\Delta_\rho u)+\lambda (f-u)=0, & \mbox{in } \Omega, \\
    \frac{\partial(|\Delta_\rho u|^{p-2}\Delta_\rho u)}{\partial\vec{n}}=0, & \mbox{on } \partial\Omega, \\
    \frac{\partial u}{\partial\vec{n}}=0, & \mbox{on } \partial\Omega.
  \end{cases}
\end{align}
A function $u\in W$, where
\begin{equation*}
  W=\{v\in W^{2,p}(\Omega),~~ \frac{\partial v}{\partial\vec{n}}=0 ~~\textrm{on}~~ \partial\Omega\},
\end{equation*}
is called a weak solution of \eqref{CP} if it
 satisfies
\begin{equation}\label{CP11}
  -\int_\Omega|\Delta_\rho u|^{p-2}\Delta_\rho u \Delta_\rho \varphi \rho dx+\lambda\int_\Omega (f-u)\varphi \rho dx=0,
\end{equation}
for any $\varphi\in W$.
It can be verified by the direct method in the calculus of variations, as did in the proof of Lemma \ref{th:4.1}, that
\begin{equation*}
  \frac{1}{p}\int_\Omega|\Delta_\rho u|^p \rho dx+\frac{\lambda}{2}\int_\Omega |u-f|^2 \rho dx
\end{equation*}
admits a unique minimizer $u\in W$, which is also the unique weak solution of equation \eqref{CP}.

\begin{theorem}\label{th:4.2}
  Let $u_n$ be a solution of the $p$-biharmonic equation \eqref{P} on graph $G_n$ and $u$ be a weak solution of equation \eqref{CP}. Then
   \begin{equation}\label{eq:th:3.4:00a}
     \lim_{n\rightarrow\infty}(\mu_n,u_n)=(\mu,u),\quad \textrm{in}~~ TL^p(\Omega),
   \end{equation}
   almost surely.
\end{theorem}

\begin{proof}
Lemma \ref{le:2.1} indicates that
there exists a sequence of transportation maps $\{T_n\}_{n=1}^\infty$ from $\mu$ to $\mu_n$. By Remark \ref{re:1}, conclusion \eqref{eq:th:3.4:00a} is equivalent to
  \begin{equation}\label{eq:th3.4:0}
    \lim_{n\rightarrow\infty}\int_{\Omega}|\tilde{u}_n(x)-u(x)|^{p}dx=0.
  \end{equation}
  where $\tilde{u}_n=u_n\circ T_n$.

  1. The transportation map $T_n$ allows us to rewrite the estimate \eqref{eq:th3.3} as
\begin{equation}\label{eq:th3.4:1a}
  \int_\Omega\left|\Delta^{\tilde{\eta}}_{\rho,\varepsilon_n}\tilde{u}_n(x) \right|^p\rho(x)dx
  +\frac{\lambda}{2}\int_\Omega|\tilde{u}_n(x)|^2\rho(x)dx
  \leq C.
\end{equation}
It follows from Corollary \ref{th:3.4} that
\begin{equation}\label{eq:th3.4:1b}
  \int_\Omega |\tilde{u}_n|^pdx+\int_{\Omega}\int_{\Omega}{\zeta}_{\varepsilon_n}(|x-y|)\left|\frac{\tilde{u}_n(y)-\tilde{u}_n(x)}{\varepsilon_n}\right|^pdydx\leq C,
\end{equation}
if $1<p< 2$, and
\begin{align}\label{eq:th3.4:1c}
\begin{split}
  &\int_\Omega |\tilde{u}_n|^pdx
  +\int_\Omega\int_\Omega\zeta_{\varepsilon_n}(|x-y|)\left|\frac{\tilde{u}_n(y)-\tilde{u}_n(x)}{\varepsilon_n}\right|^2dydx\\
  &+\int_\Omega\int_\Omega\zeta_{\varepsilon_n}(|x-y|)\left|\frac{(|\tilde{u}_n|^{(p-2)/2}\tilde{u}_n)(y)-(|\tilde{u}_n|^{(p-2)/2}\tilde{u}_n)(x)}{\varepsilon_n}\right|^2dydx
  \leq C,
\end{split}
\end{align}
if $p\geq2$.

We conclude from estimates \eqref{eq:th3.4:1a}--\eqref{eq:th3.4:1c}
and the $L^p$ boundedness of $\Delta^{\tilde{\eta}}_{\rho,\varepsilon_n}\tilde{u}_n$, i.e.,
    \begin{equation*}
      \int_\Omega\left||\Delta^{\tilde{\eta}}_{\rho,\varepsilon_n}\tilde{u}_n|^{p-2}\Delta^{\tilde{\eta}}_{\rho,\varepsilon_n}\tilde{u}_n \right|^{\frac{p}{p-1}}dx
      =\int_\Omega\left|\Delta^{\tilde{\eta}}_{\rho,\varepsilon_n}\tilde{u}_n(x) \right|^pdx\leq C
    \end{equation*}
that
there exists a subsequence of $\{\tilde{u}_n\}$ (still denoted by itself), functions $u\in L^2(\Omega)\cap L^p(\Omega)$, $\theta\in L^p(\Omega)$, and $\vartheta\in L^{p/(p-1)}(\Omega)$, such that
\begin{align}\label{coverge:1}
  \tilde{u}_n&\rightharpoonup u,\quad \textrm{in}~ L^2(\Omega)\cap L^p(\Omega),\\\label{coverge:2}
  \Delta^{\tilde{\eta}}_{\rho,\varepsilon_n}\tilde{u}_n + \tilde{u}_n&\rightharpoonup \theta,\quad \textrm{in}~ L^p(\Omega), \\\label{coverge:3}
  |\Delta^{\tilde{\eta}}_{\rho,\varepsilon_n}\tilde{u}_n|^{p-2}
  \Delta^{\tilde{\eta}}_{\rho,\varepsilon_n}\tilde{u}_n&\rightharpoonup
  \vartheta,\quad \textrm{in  } L^{p/(p-1)}(\Omega).
\end{align}
Moreover, we claim that $u\in W^{1,\min\{p,2\}}(\Omega)$ and
\begin{equation}\label{eq:LP}
  \tilde{u}_n\rightarrow u,\quad  \textrm{in}~ L^p(\Omega).
\end{equation}
In fact,
 Lemma \ref{le:2.3}, \eqref{eq:th3.4:1b}--\eqref{eq:th3.4:1c}, and \eqref{coverge:1} imply that $u\in W^{1,\min\{p,2\}}(\Omega)$ and $\tilde{u}_n\rightarrow u$ in $L^{\min\{p,2\}}(\Omega)$ (up to a subsequence).
 To prove \eqref{eq:LP} for $p>2$,
we further extract a subsequence (still denoted by itself) such that $\tilde{u}_n\rightarrow u$ a.e. in $\Omega$.
Notice that
\begin{equation*}
  \int_\Omega\left| |\tilde{u}_n|^{(p-2)/2}\tilde{u}_n \right|^2dx
  =\int_\Omega|\tilde{u}_n|^pdx\leq C.
\end{equation*}
It follows from \cite[Chapter II, Lemma 2.3]{ladyzhenskaia1968linear} that
\begin{equation*}
  |\tilde{u}_n|^{(p-2)/2}\tilde{u}_n
  \rightharpoonup |u|^{(p-2)/2}u,\quad \textrm{in}~ L^2(\Omega).
\end{equation*}
We utilize \eqref{eq:th3.4:1c} and apply Lemma \ref{le:2.3} once again for $|\tilde{u}_n|^{(p-2)/2}\tilde{u}_n$ to obtain that
\begin{equation*}
  |\tilde{u}_n|^{(p-2)/2}\tilde{u}_n
  \rightarrow |u|^{(p-2)/2}u,\quad \textrm{in}~ L^2(\Omega).
\end{equation*}
Namely,
\begin{align*}
  0&=\lim_{n\rightarrow\infty}\int_\Omega\left| |\tilde{u}_n|^{(p-2)/2}\tilde{u}_n -|u|^{(p-2)/2}u\right|^2dx\\
  &=\lim_{n\rightarrow\infty}
  \int_\Omega |\tilde{u}_n|^pdx
  +\int_\Omega |u|^pdx
  -2\lim_{n\rightarrow\infty}\int_\Omega|\tilde{u}_n|^{(p-2)/2}\tilde{u}_n  |u|^{(p-2)/2}udx\\
  &=\lim_{n\rightarrow\infty}
  \int_\Omega |\tilde{u}_n|^pdx
  -\int_\Omega |u|^pdx.
\end{align*}
This together with \eqref{coverge:1} imply \eqref{eq:LP}.

Since equation \eqref{CP} admits a unique solution, the whole sequence $\{\tilde{u}_n\}$ (not only a subsequence) converges to $u$.
This verifies \eqref{eq:th3.4:0}.
To finish the proof, we only need to show that $u\in W$ and $u$ satisfies \eqref{CP11}.

2.
Let us now prove $u\in W$ and identify $\theta$. For any $\varphi\in C^\infty(\overline{\Omega})$ with $\frac{\partial\varphi}{\partial\vec{n}}=0$ on $\partial\Omega$,
Lemma \ref{le:3.2} and \eqref{coverge:2} yield
% it's possible to use the weak version of Lemma 3.2
\begin{align*}
  \int_\Omega \left(\Delta^{\tilde{\eta}}_{\rho,\varepsilon_n}\tilde{u} _n+\tilde{u} _n\right)\cdot{\varphi}\rho dx
  &=\int_\Omega\tilde{u} _n\rho\cdot\Delta^{\tilde{\eta}}_{\rho,\varepsilon_n}{\varphi} dx + \int_{\Omega}\tilde{u} _n {\varphi}\rho dx\\
  &\rightarrow\int_\Omega  u\rho\cdot\Delta_{\rho}\varphi dx
 + \int_{\Omega}u \varphi\rho dx\\
  &=
  -\frac{\sigma_\eta}{2}\int_\Omega \rho^2\nabla u\cdot\nabla\varphi dx+ \int_{\Omega}u \varphi\rho dx=\int_\Omega\theta\varphi\rho dx.
\end{align*}
% Given $\theta\in L^p(\Omega)$, there exists a function $v\in W^{2,p}(\Omega)$, which is the unique solution of $\Delta_\rho u+u=f$ (with the homogeneous Neumann boundary condition). By the above result,and the uniqueness of $u$ and $v$, we can see that $v=u$.
This means that $u$ is a weak solution of
\begin{align*}
  \begin{cases}
    \Delta_\rho u + u =\theta, & \mbox{in } \Omega,\\
    \frac{\partial u}{\partial \vec{n}}=0, & \mbox{on  } \partial\Omega.
  \end{cases}
\end{align*}
Under assumption (A6), the $L^p$ theory for elliptic equations \cite[Theorem 2.4.2.7]{grisvard2011elliptic} implies that $u\in W$ and $\theta=\Delta_\rho u +u$ a.e. in $\Omega$.

3. Now we pass to the limit $n\rightarrow \infty$ for equation \eqref{P} to prove that $u$ satisfies \eqref{CP11}.
By the transportation map $T_n$, it can be rewritten as
\begin{align*}
  -\Delta^{\tilde{\eta}}_{\rho,\varepsilon_n}
  (|\Delta^{\tilde{\eta}}_{\rho,\varepsilon_n}\tilde{u}_n|^{p-2}
  \Delta^{\tilde{\eta}}_{\rho,\varepsilon_n}\tilde{u}_n)(x)
   +\lambda(\tilde{f}_n(x)-
  \tilde{u}_n(x)) =0, \quad x\in\Omega.
\end{align*}
Let $\varphi\rho$, where $\varphi\in W$, be the test function. We have
% \begin{align*}
%   -\sum_{i=1}^{n} \Delta_{G_n}
%   (| \Delta_{G_n}u_n|^{p-2} \Delta_{G_n}u_n)(x_i)\varphi(x_i)
%   +\lambda\sum_{i=1}^{n}(f(x_i)-u_n(x_i))\varphi(x_i)=0.
% \end{align*}
% By the transportation map $T_n$, it can be rewritten as
\begin{align*}
  -\int_{\Omega}\Delta^{\tilde{\eta}}_{\rho,\varepsilon_n}
  (|\Delta^{\tilde{\eta}}_{\rho,\varepsilon_n}\tilde{u}_n|^{p-2}
  \Delta^{\tilde{\eta}}_{\rho,\varepsilon_n}\tilde{u}_n)
  {\varphi}\rho dx+\lambda\int_{\Omega}(\tilde{f}_n-
  \tilde{u}_n){\varphi}\rho dx=0,
\end{align*}
which is equivalent to, by nonlocal integration by parts,
\begin{align}\label{eq:th3.4:3}
  -\int_{\Omega}|\Delta^{\tilde{\eta}}_{\rho,\varepsilon_n}
  \tilde{u}_n|^{p-2}\Delta^{\tilde{\eta}}_{\rho,\varepsilon_n}
  \tilde{u}_n
  \cdot\Delta^{\tilde{\eta}}_{\rho,\varepsilon_n}{\varphi}\rho dx+\lambda\int_{\Omega}(\tilde{f}_n-
  \tilde{u}_n){\varphi}\rho dx=0.
\end{align}
We utilize \eqref{coverge:3}, Lemma \ref{le:3.2},  and assumption (A2) to pass to the limit to find
\begin{equation}\label{eq:th3.4:4}
  -\int_\Omega\vartheta\Delta_\rho\varphi\rho dx+\lambda\int_\Omega(f-u)\varphi\rho dx=0.
\end{equation}
We are left to show that
\begin{equation}\label{eq:th3.4:5}
  \int_\Omega\vartheta\Delta_\rho\varphi\rho dx=
  \int_\Omega|\Delta_\rho u|^{p-2}\Delta_\rho u\Delta_\rho\varphi\rho dx.
\end{equation}
By the monotonicity of $|s|^{p-2}s$,
\begin{equation*}
  \int_\Omega\left(|\Delta^{\tilde{\eta}}_{\rho,\varepsilon_n}\tilde{u}_n|^{p-2}
  \Delta^{\tilde{\eta}}_{\rho,\varepsilon_n}\tilde{u}_n
  -|\Delta_\rho v|^{p-2}\Delta_\rho v\right)
  (\Delta^{\tilde{\eta}}_{\rho,\varepsilon_n}\tilde{u}_n-\Delta_\rho v)\rho dx\geq 0,
\end{equation*}
holds for any $v\in W$.
Combining it with \eqref{eq:th3.4:3} for ${\varphi}=\tilde{u}_n$ leads us to
\begin{align*}
\int_\Omega\left(|\Delta^{\tilde{\eta}}_{\rho,\varepsilon_n}\tilde{u}_n|^{p-2}
\Delta^{\tilde{\eta}}_{\rho,\varepsilon_n}\tilde{u}_n
  \right)
  \Delta_\rho v\rho dx+
  \int_\Omega\left(|\Delta_\rho v|^{p-2}\Delta_\rho v\right)
  (\Delta^{\tilde{\eta}}_{\rho,\varepsilon_n}\tilde{u}_n-\Delta_\rho v)\rho dx\\
  -\lambda\int_{\Omega}(\tilde{f}_n-
  \tilde{u}_n)\tilde{u}_n\rho dx\leq0.
\end{align*}
Let $n\rightarrow\infty$, we obtain
\begin{align*}
  \int_{\Omega}\vartheta\Delta_\rho v \rho dx
  +\int_\Omega |\Delta_\rho v|^{p-2}\Delta_\rho v\Delta_\rho(u-v)\rho dx
  -\lambda\int_\Omega(f-u)u\rho dx\leq 0.
\end{align*}
Substituting it into \eqref{eq:th3.4:4} with $\varphi=u$ to obtain
\begin{equation*}
  \int_\Omega(|\Delta_\rho v|^{p-2}\Delta_\rho v-\vartheta)\Delta_\rho(u-v)\rho dx\leq 0.
\end{equation*}
Taking $v=u-\alpha\varphi$ in the above with $\alpha>0$
and letting $\alpha\rightarrow 0^+$, we have
\begin{equation*}
  \int_\Omega(|\Delta_\rho u|^{p-2}\Delta_\rho u-\vartheta)\Delta_\rho\varphi\rho dx\leq 0.
\end{equation*}
An opposite inequality is obtained by choosing $\alpha<0$ and letting $\alpha\rightarrow 0^-$.
Consequently,
\eqref{eq:th3.4:5} holds and $u$ satisfies \eqref{CP11}.
\end{proof}

% \begin{remark}
%   The only place where assumption (A4) (i.e., assumption \eqref{eq:Tn}) is utilized in the proof of Theorem \ref{th:4.2} is when formula \eqref{eq:remark3.3:1} is quoted.
%   By remark \ref{remark:3.3},
%   if we consider piecewise constant $\eta$, conclusion \eqref{eq:th:3.4:00a} holds for any $\varepsilon_n\gg \delta_n^{{d}/{(d+1)}}$.
% \end{remark}

\begin{remark}
  According to $\eqref{eq:hyper}$,
 the hypergraph $p$-Laplacian equation
  \begin{equation}\label{eq:hypereq}
    \textrm{div}_{H_n} (|\nabla_{H_n} u|^{p-2} \nabla_{H_n} u)(x_i)
    +\lambda(f(x_i)-u(x_i))=0,\quad x_i\in\Omega_n,
  \end{equation}
  is equivalent to the $p$-biharmonic equation \eqref{P} on graph $G'_n$.
  Here we rescale both $\nabla_{H_n}$ and $\textrm{div}_{H_n}$ with the constant $\frac{1}{n\varepsilon_n^2}$.
  Recall that $G_n'$ has the same structure as $G_n$ except that all weights in $G_n'$ are set to one. This corresponds to  $\eta=\chi_{[0,1]}$.
  Consequently, the conclusion \eqref{eq:th:3.4:00a} also holds for equation \eqref{eq:hypereq}.
  % under assumption $\varepsilon_n\gg \delta_n^{{d}/{(d+1)}}$.
\end{remark}

%The acknowledgments section should not be numbered.
\section*{Acknowledgments}
% This work was finished during KS's visit to DESY funded by the China Scholarship Council.
The authors would like to thank the referee for the valuable comments and suggestions.
KS is supported by China Scholarship Council. The authors acknowledge support from DESY (Hamburg, Germany), a member of the Helmholtz Association
HGF.

%%%%%%%%%%%%%%%%%%%%%%%%%%%%%%%%%%%%%%%%%%%%%%%%%%%%%%

%\begin{thebibliography}{99}

%
%\end{thebibliography}

\bibliographystyle{unsrt}
\bibliography{references}

\end{document}